\documentclass{amsart}
\usepackage{amsmath, amssymb,hyperref}
\usepackage{tikz,enumerate}
\allowdisplaybreaks

\newcommand{\RR}{\mathbb R}

\newcommand{\EE}{\mathbb E}
\newcommand{\PP}{\mathbb P}
\newcommand{\C}{\mathcal C}
\newtheorem{example}{Example}
\newtheorem{proposition}{Proposition}
\newtheorem{theorem}{Theorem}

\begin{document}

\title{Properties of a random Cantor set with overlaps}
\author{Anna Chiara Lai and Paola Loreti}
\date{}
\begin{abstract}
    We study the topology and the Hausdorff dimension of a random Cantor set with overlaps, generated by an iterated function system with scaling ratio equal to the Golden Mean. The results extend known formulas to a case where the Open Set Condition fails. Our methodology is based on the theory of expansions in non-integer bases. 
\end{abstract}
\maketitle
Readers are invited to consult the revised version of this paper published in\\
Lai, A.C., Loreti, P. Properties of a random Cantor set with overlaps. Res. number theory \textbf{12}, 71 (2026).

\section{Introduction}

In this paper, we investigate the Hausdorff dimension of a random Cantor set, via the theory of expansions in non-integer bases.

\medskip We focus on a one-dimensional case study related to the non-integer base binary expansions of real numbers. More precisely, we consider the Iterated Function System (IFS) generated by the contractive maps 
$$f_i(x):=\frac{1}{\varphi}(x+i), \qquad i=0,1$$
where $\varphi$ is the golden mean. The choice of the particular base $\varphi$ is motivated by its prominent role in the theory of expansions in non integer bases: the uniqueness phenomena in the representation that are the core of  OSC-based arguments begin to emerge only when $q>\varphi\sim 1.6180$ \cite{golden}, and a continuum of unique expansions is ensured only above the Komornik-Loreti constant $q_{KL}\sim 1.78723$ \cite{klconstant}. This IFS is widely investigated in the theory of expansions in non-integer bases introduced in \cite{renyi57}, because its iteration provides expansion of any $x\in[0,\varphi]$ of the form 
$$x=\sum_{i=1}^\infty \frac{c_i}{\varphi^i}.$$
In particular for all $n$ and for all  $\mathbf c=c_1c_2\cdots c_n\in\{0,1\}^n$ we have the identity
$$I_{\mathbf c}:=\left[\sum_{i=1}^n \frac{c_i}{\varphi^i},\sum_{i=1}^n \frac{c_i}{\varphi^i}+\frac{1}{\varphi^{n-1}}\right]=f_{c_1}\circ\cdots\circ f_{c_n}([0,\varphi]).$$
Starting from $\mathcal I_0=\{[0,\varphi]\}$, we iteratively construct the family of intervals as following. For each interval $I\in \mathcal I_n$ and for each $i=0,1$, we consider the interval $f_i(I)$. The interval $f_i(I)$ is retained with probability $p$ and discarded with probability $1-p$. The collection $\mathcal I_{n+1}^p$ is then given by 
$$\mathcal I^p_{n+1}:=\{f_i(I)\mid f_i(I) \text{ is retained},\, i=0,1,\, I\in \mathcal I_n\}.$$
From this, we can define the random Cantor prefractal $\mathcal C_n^p=\cup_{I\in \mathcal I_n^p} I$ and the Cantor fractal $$\mathcal C^p:=\bigcap_{n\geq 0} \mathcal C^p_n.$$
Our goal is to investigate the expected topology and dimension of $\mathcal C^p$ as a function of the probability parameter $p$. 
The first notion of fractal dimension was proposed by Hausdorff \cite{hausdorff1918dimension,edgar2019classics}. Since then, many other definitions have been introduced, among them we make use of the box counting dimension \cite{bouligand1928ensembles}.
Our main result states that the expected Hausdorff dimension of $\mathcal C^p$ is provided by the closed formula
\begin{equation}\label{exp}\EE(\dim (\mathcal C^p))=\begin{cases}
    0&\text{if }p\in[0,1/2]\\
    \dfrac{\log(2p)}{\log \varphi}&\text{if }p\in[1/2,\varphi/2]\\
    1&\text{if }p\in[\varphi/2,1].\end{cases}\end{equation}

The above formula generalizes known results holding in the case in which the \emph{Open Set Condition} (OSC) holds. More precisely, we have from \cite{fal86} that if the scaling ration $\varphi$ is replaced by a real number $q\geq 2$ and we denote by $\mathcal C^{p,q}$ the resulting random Cantor set, we have
\begin{equation}\label{exp}\EE(\dim (\mathcal C^{p,q}))=\begin{cases}
    0&\text{if }p\in[0,1/2]\\
    \dfrac{\log(2p)}{\log q}&\text{if }p\in[1/2,1].\end{cases}\end{equation}


An IFS $\mathcal G=\{g_1,\dots, g_m\}$ satisfies the OSC if there exists an open set $X$ such that  $\mathcal G(X)\subseteq X$
and the sets $g_1(X),\dots,g_m(X)$ are pairwise disjoint \cite{moran}. Note that $\{f_0,f_1\}$ does not satisfy the OSC, because the similarity dimension $\log(2)/\log(\varphi)$   \cite{fal88} of its unique compact fixed point is greater than $1$, which is the dimension of the domain of $\{f_0,f_1\}$ \cite{schief}.

  A novelty of the formula \eqref{exp} is that in the non-OSC scenario a new probability threshold $p_d:=\varphi/2<1$  emerges: if $p_d<p\leq 1$ then $\EE(\dim (\mathcal C^p))=1$, even if almost surely $\mathcal C^p$ has no connected components.

The difficulty in the analysis lies in the fact that the well-known redundancy phenomena in non-integer base demand to estimate the expected number of survived intervals explicitly.  Indeed, at any level $n$, several distinct binary sequences $c_1\cdots c_n$ may correspond to the same interval in the prefractal $\mathcal C_n^p$, because its endpoints can be represented by different binary expansions in base $\varphi$ -- in this case we say that the underlying IFS has \emph{exact overlap}, the definition is explicited below in the Introduction. Also, even when distinct, the lack of Open Set Condition also implies that two intervals may overlap.  We selected the greedy expansions as representatives of such equivalence classes of intervals and, by applying some results on the theory of expansions in non-integer bases, we were able to quantify both the expected number of intervals in function of $p$ as well as the size of the possible overlaps. 

 Addressing random IFS not satisfying the OSC may offer quantitative and qualitative improvements of the model. Indeed, on the one hand, a larger number of systems allows more flexibility in tuning the parameters; on the other hand, allowing prefractals to intersect means describing more precisely local density phenomena. 

\medskip
\emph{State of the art.} The theoretical background of the present work goes back to Moran \cite{moran}, who first introduced the Open Set Condition. The 1980's  knew a rising interest in fractals, due to Mandelbrot's celebrated book \cite{mandelbrot} and the formalization of the IFSs and the Hutichinson's theorem, characterizing their attractor \cite{hutchinson}. A few years later, random versions of self-similar sets emerging from linear IFSs  were introduced, in connection with the percolation theory on random trees -- see the papers by Kenneth Falconer \cite{fal86,fal89} on the dimension of random Cantor sets and of their projection  in the higher dimensional settings. The morphology of random Cantor sets, and the idea that there exists some threshold probabilities that shift the geometry of the random Cantor set, were formalized in  \cite{dek90}.  Under the assumption of the OSC, higher-dimensional random Cantor sets were further investigated in \cite{falgri92}. The recent survey paper \cite{survey25} provides a wide overview of the evolution of the research on this topic in the subsequent decades. We focus now on more recent results on the case without OSC, which is the object of our study. The paper \cite{hochman} addresses the dimensionality problem of iterated function systems with overlaps in the deterministic case. We recall that an IFS of the form $\{g_{a_i}(x):=(x+a_i)/q\mid a_i\geq 0, i=1,\dots,m\}$  is said to contain an exact overlap
if there exists $c_1,\dots,c_n,d_1,\dots,d_n  \in \{a_i\}_{i=1}^m$, for some $n\geq 0$, such that $g_{c_n}\circ\cdots \circ g_{c_1}(x) = g_{d_n}\circ\cdots \circ g_{d_1}(x)$ and $ c_j\not=d_j$ for some $j=1,\dots,n$. In other words, an exact overlap is given when two or more intervals in the prefractal coincide, determining a discrepancy between the similarity dimension and the Hausdorff dimension of the fractal set, a phenomenon also occurring in our model. Recent studies on the dimensionality of IFS with exact overlap can be found in \cite{bak23,rap22,baker24} and in \cite{hoc22JEMS}: the approach in these papers is measure theoretic, while here we adopt a combinatorial methodology. A random version of such problem for affine IFS was investigated in \cite{guo}, where a condition (the Finite Intersection Property) allowing overlaps and weaker than the OSC was investigated. 
For random fractals with overlaps emerging from  non-integer base expansions we refer to \cite{daj24,kem16} and in the references therein. 
Our methodology is based on counting the exact overlaps, moving the problem to a combinatorial study on the cardinality of finite expansions in base $\varphi$.  This approach roots in the seminal paper by \cite{renyi57}, where expansions in non-integer bases were first introduced, and in the subsequent studies on the redundancy of the representations in non-integer base \cite{golden,gle01}. In the case of overlapping intervals, we consider a representative which is given by the \emph{greedy expansion} of its left endopoint. The digit constraints on greedy expansions with the golden mean as base (and, in general Pisot, numbers) are at the core of our argument; they were first investigated in \cite{parry1960} -- related papers are \cite{frougny1992,akiyama04}. The problem of counting the number of finite expansions in base $\varphi$ was recently addressed, among others, in \cite{dekking2024}, here we proof recursive formulas that better fit our arguments.

\medskip

\medskip



\medskip

\medskip
\emph{Organization of the paper.} In Section \ref{s2} we introduce the random Cantor set under analysis and we present its main properties. Section \ref{s3} contains the statement and the proof of our main result. In Section \ref{s4} we draw our conclusions. The Appendix \ref{app1} aims to make the paper as self-contained as possible by presenting some original and technical proofs of earlier results.

\section{A random Cantor set with overlaps and its properties}\label{s2}

As anticipated in the Introduction, we consider the contractive linear maps
$$f_i(x):=\frac{1}{\varphi}(x+i), \quad i\in\{0,1\}$$
where $\varphi:=\frac{1+\sqrt{5}}{2}$ is the golden mean, namely the greatest solution of the equation $\varphi^2=\varphi+1$. Consider the Iterated Function System (IFS) $\mathcal F:=\{f_0,f_1\}$ and the associated Hutchinson operator
$$\mathcal F(X):=\bigcup_{i=0,1}f_i(X)$$
defined for all $X\subset \RR$. 
Let the initial compact set be the interval
$$I_0:=\left[0,\frac{1}{\varphi-1}\right]=[0,\varphi].$$
It is easy to check that $\mathcal F(I_0)=I_0$ and, consequently, $I_0$ is the attractor of $\mathcal F$. We need a finer description of the iteration of $\mathcal F$. To describe the iteration more precisely, we define for all $n$ and for all  sequences $\mathbf c=c_1\dots c_n\in\{0,1\}^n$
the interval 
$$I_\mathbf c:=f_{c_n}\circ\cdots\circ f_{c_1}(I_0)=\left[\sum_{i=1}^n \frac{c_{n-i}}{\varphi^i},\sum_{i=1}^n\frac{ c_{n-i}}{\varphi^i}+\frac{1}{\varphi^{n-1}}\right].$$
Consider the collection of intervals for the deterministic case $\mathcal I_n=\{I_{\mathbf c}\mid \mathbf c\in \{0,1\}^n\}$. 
Note that elements of $\mathcal I_n$ may intersect 
(for instance $f_0(I_0)\cap f_1(I_0)\not=\emptyset$) 
or even overlap (after a few computation one can check, for instance, that $I_{110}= I_{001}$), see Figure \ref{fig1}. 
In case of overlapping, only one representative is included in $\mathcal I_n$, so that in general $\#\mathcal I_n<2^n$. We point out  that  $\# \mathcal I_n\sim \varphi^n$ -- see Proposition \ref{p1} below. 

\begin{figure}[h!]\centering 
\includegraphics[width=\textwidth, height=0.4\textheight]{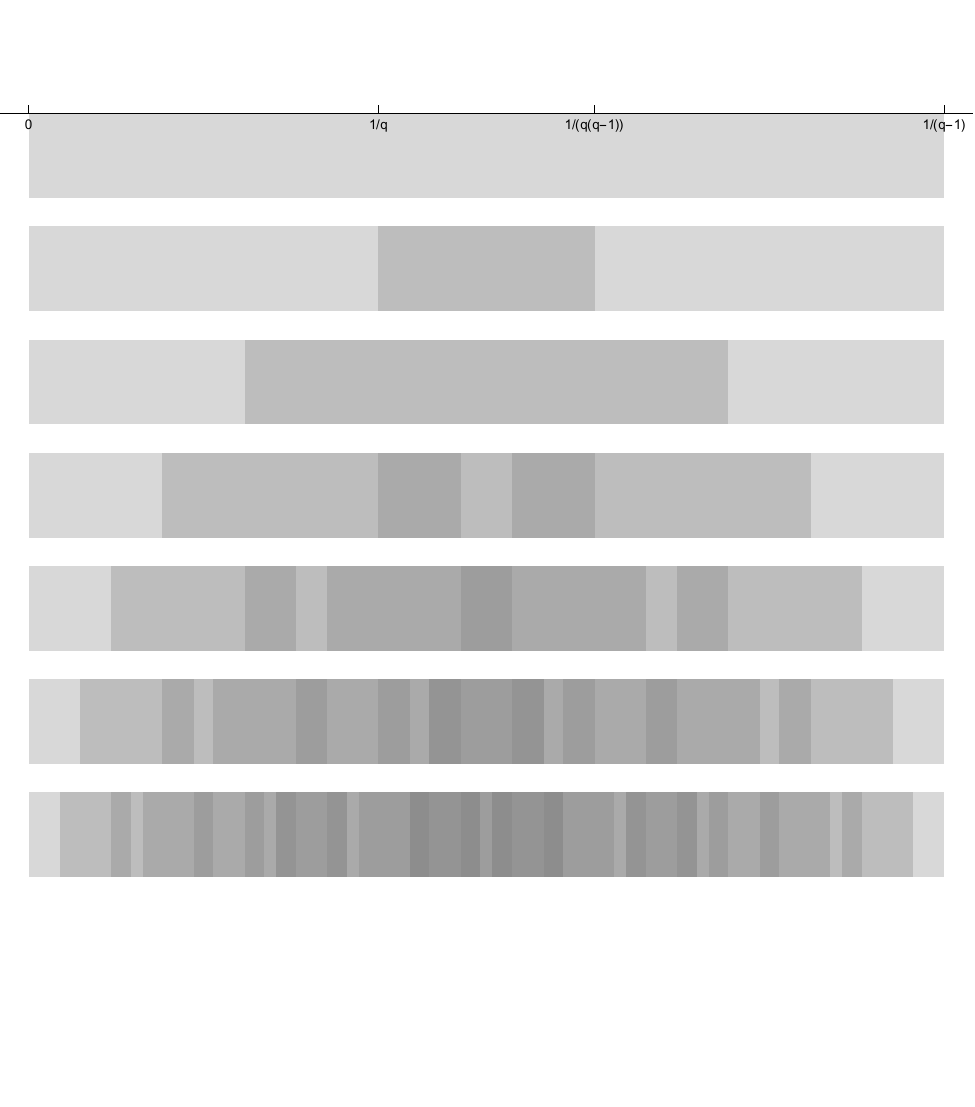}
\vskip-1cm
\caption{\label{fig1}
Band visualization of the deterministic case. Each horizontal band corresponds to the union of the intervals in $\mathcal I_n$, for $n=0,\dots,6$. Interval intersections are highlighted by shading. For instance the third band corresponds to $\mathcal I_2=\{[0,1/\varphi],[1/\varphi,2/\varphi],[1/\varphi^2,1],[1,\varphi]\}$, the darker central area is the union of the pairwise, disjoint intersections $[0,1/\varphi]\cap [1/\varphi^2,1]$, $[1/\varphi,2/\varphi]\cap [1/\varphi,1]$ and $[1/\varphi^2,1]\cap [1/\varphi,2/\varphi]$. }
\end{figure}

To randomize the process, we consider the collection $\mathcal I_n^p$ (with fixed $p\in(0,1]$) recursively generated from $I_0^p=[0,\varphi]$ by the following Bernoulli process. 
For each $n\geq 1$,  the IFS is applied to each element $I_\mathbf c$ of $\mathcal I_n^p$, namely the sets
$I_{\mathbf cc_{n+1}}=f_{c_{n+1}}(I_\mathbf c)$ with $c_{n+1}\in\{0,1\}$ are considered. Each of them is retained with probability $p$ and discarded with probability $1-p$, so that 
$$\mathcal I_{n+1}^{p}=\{I_{\mathbf {c}{c_{n+1}}}\mid I_\mathbf c\in \mathcal I_n^p \text{ and } f_{c_{n+1}}(I_\mathbf c) \text{ is retained }\}.$$
In other words, $\mathcal I_{n}^p$ contains all the intervals $f_{\mathbf c}(I_0)$ where $\mathbf c$ are the branches of a random binary tree with uniform probability $p$ and depth $n$. The resulting prefractal is 
$$\mathcal C_n^p:=\bigcup_{I\in \mathcal I_n^p} I,$$
see Figure \ref{fig2}.
\begin{figure}[h!]\centering
\includegraphics[width=\textwidth, height=0.4\textheight]{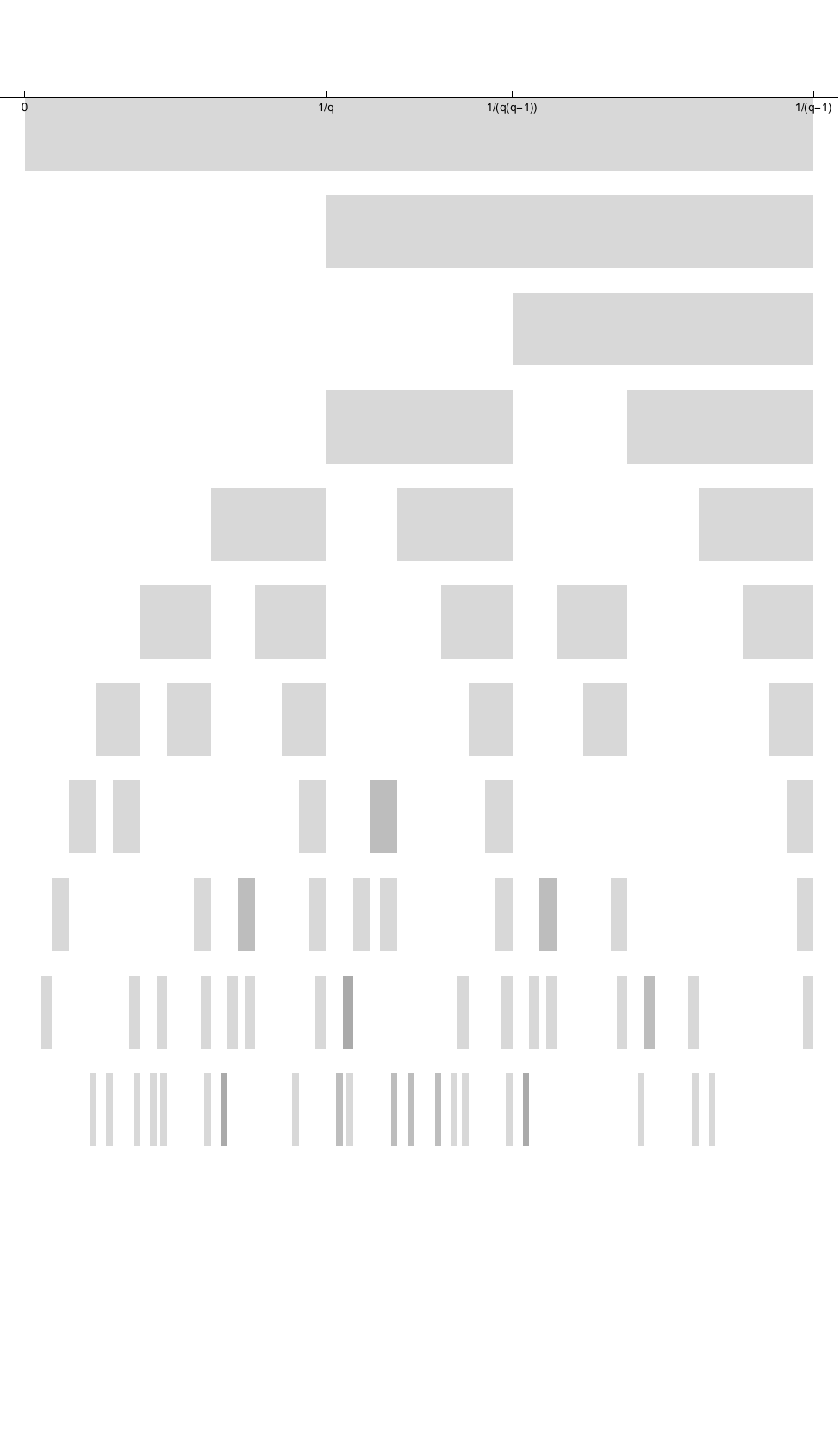}  \vskip-1cm
\caption{\label{fig2}
Band visualization of a random prefractal with $p=0.7$. Each horizontal band corresponds to $\mathcal C^p_n$ for $n=0,\dots,10$. Interval intersections are highlighted by shading. 
}
\end{figure}

The associated random Cantor set is
$$\mathcal C^p:=\bigcap_{n\geq 0} \mathcal C_n^p.$$
Note that the above definition is well posed since $f_{c}(I_\mathbf c)\subset I_\mathbf c$ for all $n\geq 0, \mathbf c\in\{0,1\}^n$ and $c\in\{0,1\}$. Hence $\mathcal C_{n+1}^p\subseteq C_{n}^p$ for all $n$ and this monotonicity implies the existence of the (possibly empty) limit set $\mathcal C^p$. 





The lack of OSC in our system, showed in the Introduction, requires an ad hoc analysis of the combinatorics of $\mathcal I_n^p$ and on the measure of $\mathcal C^p_n$. 
As remarked above the map
\begin{equation}\label{map}
    \mathbf c\mapsto I_\mathbf c=\left[\sum_{i=1}^{|\mathbf c|}\frac{c_{|\mathbf c|-i}}{\varphi^i},\sum_{i=1}^{|c|}\frac{c_{|\mathbf c|-i}}{\varphi^i}+\frac{1}{\varphi^{|\mathbf c|-1}}\right]
\end{equation}
is not injiective. We introduce the equivalence class
$$\mathbf c \sim \mathbf d \quad \Leftrightarrow \quad I_\mathbf c=I_\mathbf d$$
we denote by $[\mathbf c]$ be the equivalence class of $\mathbf c$ and let $\{0,1\}^n/\sim$ be the set of such classes. We can choose as class representatives the \emph{greedy expansions}. Classically, greedy expansion represent the lexicographic greatest expansion $(c_i)_{i\geq 1}$ of a number $x=\sum_{i=0}^\infty c_i q^{-i}$, given the non-integer bases $q$. In the particular case $q=\varphi$, where $\varphi$ is the golden mean, greedy expansions are binary, possibily finite (namely ending with $0^\infty$), sequences characterized by the non-occurrence of the forbidden word $011$ \cite{parry1960}. The fact that greedy expansions are eligible as class representatives follows from the fact that any binary expansion of length $n$ can be converted into a greedy expansion of the same number with the same length $n$  \cite{frou1,frou2} -- note that it suffices to replace any occurrence of $011$ of a non-greedy expansion with $100$, to get a greedy expansion of the same number.
 For $\mathbf c=c_1\cdots c_n\in\{0,1\}^n$, define $rev ({\mathbf c}):=c_n c_{n-1}\cdots c_1$, so that setting
$$\mathbf G_n:=\{\mathbf c\in\{0,1\}^n\mid \mathbf c \text{ is greedy}\} $$
we  have in the deterministic case
$$\mathcal I_n=\{f_{rev({\mathbf c})}(I_0)\mid \mathbf c\in\mathbf G_n\}=\bigcup_{\mathbf c\in\mathbf G_n}\left[\sum_{i=1}^n\frac{c_i}{\varphi^i},\sum_{i=1}^n\frac{c_i}{\varphi^i}+\frac{1}{\varphi^{n-1}}\right].$$

\begin{example}
    For $n=3$, 
    $$\mathbf G_3=\{000,001,010,100,110,101,111\}$$
    and $\#[\mathbf c]=1$ if $\mathbf c\not=100$ and 
    $$\#[100]=\#\{100,011\}=2.$$
    \end{example}
Let 
$$\mathbf G_n^h:=\{\mathbf c\in \mathbf G_n\mid c_{n-h-1}\cdots c_n=10^h\} \quad \text{for $h\in 0,\dots,n-1$}$$
and $\mathbf G^n_{-1}:=(0)^n$
so that 
\begin{equation}\label{gn}\mathbf G_n=\bigcup_{h=-1}^{n-1} \mathbf G_n^h.\end{equation}
The first result collects some known results on $\mathcal I_n$ in the deterministic case, see \cite{Loth} for a general introduction and \cite{spect1,spect2} for more technical aspects, concerning the granularity of expansions in non-integer bases. We propose in the Appendix \ref{app1} self-contained proofs for the seek of clarity. Recall that the $n$-Fibonacci number $F_n$ is recursively defined by $F_0=0,F_1=1$ $F_{n+2}=F_{n+1}+F_{n}$ for $n\geq 0$.

\begin{proposition}\label{p1}
Let $n\geq 0$ and consider the collection $\mathcal I_n=\{I_\mathbf c=f_\mathbf c(I_0)\mid \mathbf c\in\{0,1\}^n\}$. Then
\begin{enumerate}[i)]
\item the cardinality of $\mathcal I_n$ is $F_{n+3}-1$, where $F_n$ is the $n$-th Fibonacci number; 
\item for all $I=f_\mathbf c(I_0)$ with $\mathbf c\in \{0,1\}^n$ we have $|I|=\frac{1}{\varphi^{n-1}}$;
\item for all $\mathbf c,\mathbf d\in \{0,1\}^n$ with $\mathbf c\not=\mathbf d$ 
$$|I_\mathbf c \cap I_\mathbf d|\in \left \{0,\frac{1}{\varphi^{n+1}},\frac{1}{\varphi^{n}},\frac{1}{\varphi^{n-1}}\right\}.$$
\end{enumerate} 
\end{proposition}

\begin{proof}
As mentioned above, the proof is postponed in Appendix \ref{app1}.
\end{proof}

In the non-deterministic case, in order to quantify the expected cardinality of $\mathcal I^p_n$, we need to know the cardinality $\#[\mathbf c]$ of each class $[\mathbf c]$, namely how many different binary expansion there exists for each greedy expansion.


The next results shows some recursive relation on the cardinality of each class of equivalence.

\begin{proposition}\label{p2}
Let $n\geq 0$. Then
\begin{enumerate}[(i)]
    \item If $\mathbf c 1\in \mathbf G_{n+1}^0$ then $\#[\mathbf c1]=\#[\mathbf c]$.
    \item If $\mathbf c 0\in \mathbf G_{n+1}^1$ then $\#[\mathbf c0]=\#[\mathbf c]$.
    \item If $\mathbf c\in \mathbf G_{n}^{h}$ for some $h\geq 0$ then 
$$[\mathbf c 00]\geq [\mathbf c]+ 1 $$
\end{enumerate}
\end{proposition}
\begin{proof}
We begin by proving (i). Let  $\mathbf c 1\in \mathbf G_{n+1}^0$, and remark that by definition of $\mathbf G_{n+1}^0$, the sequence $\mathbf c 1$ is a greedy expansion. Let $x:=\sum_{i=1}^{n+1} c_i/\varphi^{i}$. Now, recall from \cite{optimal} that the greedy expansions in base $\varphi$ such as $\mathbf c 1$ are optimal, in particular they minimize the reminder $x-\sum_{i=1}^n\frac{c_i}{\varphi^i}$. Therefore, if $\mathbf d d_{n+1}\in[\mathbf c 1]$, then 
$$\frac{d_{n+1}}{\varphi^{n+1}}=x-\sum_{i=1}^{n}\frac{d_i}{\varphi^i}\geq x-\sum_{i=1}^n\frac{c_i}{\varphi^i}=\frac{1}{\varphi^{n+1}}.$$
Hence $d_{n+1}=1$ and, consequently, $\mathbf d\in[\mathbf c]$. Therefore $$\#[\mathbf c 1]=\#\{\mathbf d 1\mid \mathbf d\in[\mathbf c]\}=\#[\mathbf c].$$

We proceed in proving (ii). If $\mathbf c0\in \mathbf G_{n+1}^1$ then $\mathbf c=c_1\cdots c_{n-1}10$ and $\mathbf c0$ is the greedy expansion of length $n+1$ of $x:=\sum_{i=1}^{n-1}c_i/\varphi^i+1/\varphi^{n}$. Let $d_1\cdots d_{n+1}\in[\mathbf c 0]$, so that $x=\sum_{i=1}^{n+1}d_i/\varphi^i$. As above, by the optimality of $\mathbf c0$ we deduce that 
$$\frac{1}{\varphi^{n}}\leq \frac{d_{n}}{\varphi^n}+\frac{d_{n+1}}{\varphi^{n+1}}$$
and this implies $d_n=1$. Moreover 
\begin{equation}\label{siu}\sum_{i=1}^{n}\frac{c_i}{\varphi^i}-\sum_{i=1}^{n}\frac{d_i}{\varphi^i}=\frac{d_{n+1}}{\varphi^{n+1}}\in\left\{0,\frac{1}{\varphi^{n+1}}\right\}.\end{equation}
Now, as $c_n=1$ then either $c_1\cdots c_n=(1)^n$ or $c_{n-1}=0$. We claim that in both cases $d_{n+1}=0$.
\begin{itemize} \item 
If $c_1\cdots c_{n}=(1)^n$ then  $\sum_{i=1}^{n}\frac{c_i}{\varphi^i}-\sum_{i=1}^{n}\frac{d_i}{\varphi^i}$ is either $0$ or greater than $\varphi^{-n}$. In view of \eqref{siu}, we deduce that only the first case can occur and, in particular, that $d_{n+1}=0$.
\item If $c_{n-1}=0$, then define 
$$x_c:=\sum_{i=1}^{n-1}\frac{c_i}{\varphi^i} \qquad x_d:=
\sum_{i=1}^{n-1}\frac{d_i}{\varphi^i}$$ 
Note that, since $c_n=d_n$, then $x_c-x_d=d_{n+1}/\varphi^{n+1}$.
 Assume, in order to seek a contradiction, that $d_{n+1}=1$ so that  
\begin{equation}\label{s} x_c-x_d=\frac{1}{\varphi^{n+1}}.\end{equation}
In particular, setting $\bar{\mathbf c}=c_1\cdots c_{n-1}$ 
and $\bar{\mathbf d}=d_1\cdots d_{n-1}$ we deduce that the intervals $I_{\bar {\mathbf c}}=[x_c,x_c+1/\varphi^{n-2}]$ and  $I_{\bar {\mathbf d}}=[x_d,x_d+1/\varphi^{n-2}]$ have non-empty intersection. Then, by Proposition \ref{p1}, 
 $$x_d+\frac{1}{\varphi^{n-2}}-x_c=
 |I_{\mathbf {\bar d}}\cap I_{\mathbf {\bar c}}|\in \left\{0,\frac{1}{\varphi^{n-2}},\frac{1}{\varphi^{n-1}},\frac{1}{\varphi^{n}}\right\}$$ and, consequently,
$$x_c-x_d\in\left\{0,\frac{1}{\varphi^{n-2}},\frac{1}{\varphi^{n-1}},\frac{1}{\varphi^{n}}\right\}.$$
By \eqref{s} we get the required contradiction and we conclude $d_{n+1}=0$.
\end{itemize}
As $d_{n+1}=c_{n+1}=0$, we have that 
 $d_{1}\cdots d_{n}\in[\mathbf c]$ and we conclude that 
 $$\#[\mathbf c 0]=\#\{d_1\cdots d_n0\mid  d_{1}\cdots d_{n}\in[\mathbf c]\}=\#[\mathbf c]$$
and this concludes the proof of (ii).

To prove (iii), we preliminary remark that the redundancy of the representation in base $\varphi$ emerges (only) from the substitution $100\mapsto011$. For instance, the class $[1000]$ contains 2 elements: $1000$ itself and $0110$. More generally we have
$$[10^{2m}]=\{10^{2m},0110^{2(m-1)},01010^{2(m-2)},\dots, (01)^m1\}.$$
 Similarly
 $$[10^{2m+1}]=\{10^{2m+1},0110^{2(m-1)+1},01010^{2(m-2)+1},\dots, (01)^m10\}.$$
Incindently, we remark that  
 $$\#[10^h]=1+\left\lfloor \frac{h}{2}\right\rfloor \quad \forall h\geq 0.$$
Now, clearly if $\mathbf d\in [\mathbf c]$ then $\mathbf d00\in [\mathbf c00]$. On the other hand, we define
$$\mathbf t:=\begin{cases}
    (01)^{m+1}1 \quad &\text{if }h=2m;\\
    (01)^{m}10 \quad &\text{if }h=2m-1;
\end{cases}$$
and note that $\mathbf t\in [10^{h+2}]$. 
 Let $\mathbf c=\bar{\mathbf c}10^h$ and  $\mathbf d:=\bar{\mathbf c}\mathbf t$. Then we have
$$\sum_{i=1}^{n+h+3}\frac{d_i}{\varphi^i}=\sum_{i=1}^{n}\frac{c_i}{\varphi^i}+\frac{1}{\varphi^n}\sum_{i=1}^{h+3}\frac{t_i}{\varphi^i}=\sum_{i=1}^{n}\frac{c_i}{\varphi^i}+\frac{1}{\varphi^{n+1}}=\sum_{i=1}^{n+h+3}\frac{c_i}{\varphi^i}.$$
Therefore $\mathbf d\in[c00]$. We conclude that
$$[\mathbf c00]\supseteq \{\hat {\mathbf c} 00\mid \hat {\mathbf c}\in [\mathbf c]\} \cup \{\mathbf d\}$$ 
and this implies the claim $[\mathbf c 00]\geq [\mathbf c]+ 1$.
\end{proof}

\section{Critical probabilities}\label{s3}
In this section we state and prove our main result.
\begin{theorem}
Let $p\in[0,1]$. If $p\leq 1/2$ then the random Cantor set $\mathcal C^p$ is empty almost surely.
For $p\in(1/2,\varphi/2)$ we have
$$\EE(\dim (\mathcal C^p))=\frac{\log(2p)}{\log \varphi}.$$
For $p\in(\varphi/2,1]$ we have
$$\EE(\dim (\mathcal C^p))=1.$$
However, $C^p$ has no connected components almost surely for all $p\in(0,1)$. 
\end{theorem}
\begin{proof}
The first part of the claim, namely the fact that the critical extinction probability for $\mathcal C^p$ is $1/2$, trivially follows by classical percolation theory. Indeed, we can regard the construction of $\mathcal C^p$ in terms of random binary trees. More precisely, let $\mathcal T$ be a binary tree with nodes identified by finite words $\mathbf c\in\{0,1\}^*:=\cup_{n\geq 0} \{0,1\}^n$. The root is the empty set and the two children of a node $\mathbf c$ are $\mathbf c0$ and $\mathbf c1$. Each child is retained with probability $p$ and discarded with probability $1-p$. Then the map \eqref{map} establishes a (non-injective) relation between the $n$-level $\mathcal T_n$ of  $\mathcal T$ and $\mathcal I_n^p$. In particular $\mathcal I_n^p$ is empty if and only if $\mathcal T_n$ is empty and the latter event has $\lim_{n\to+\infty} p^n= 0$ probability, see for instance \cite{SDA18,GD20}. 

\medskip
We now discuss the case $p>1/2$. 
Call $\EE_n^p$ the expected number of elements of $\mathcal I^p_n$. First we prove that 
\begin{equation}\label{e1}\EE(\dim (\mathcal C^p))=\lim_{n\to\infty }\frac{\log \EE_n}{n\log \varphi}.\end{equation}

 Define the random variable $$N_n(r):=\min \{\# S \mid S \text{ is a cover of }\mathcal C^p \text{ of diameter }r\}.$$ By Proposition \ref{p1}, the random prefractal $\mathcal C_n^p$ is a $\frac{1}{\varphi^{n-1}}$-cover of $\mathcal C^p$, and we deduce 
$$\EE\left(N_n\left(\frac{1}{\varphi^{n-1}}\right)\right)=\EE_n$$
so that, using the symbol  $\dim_{\text{BOX}}$ to denote the box counting dimension, 
\begin{align*}
    \EE(\dim (\mathcal C^p))&\leq \EE(\dim_{\text{BOX}} (\mathcal C^p))=\liminf_{r\to 0} \frac{\log \EE(N(r))}{-\log r}\\
    &\leq \lim_{n\to \infty} \frac{\log \EE_n}{n\log \varphi}.
\end{align*}

Now, to provide a lower estimate for $\EE(\dim \C^p) $ we consider the following process. First let us provide an order on the intervals in $\mathcal I^p_n=\{I_{n,i}\}$ so that the left endpoint of $I_i$ is smaller of the left endopoint of $I_{i+1}$. By Proposition \ref{p1} we have that $ I_{n,i}\cap I_{n,i+3}=\emptyset$ for all $i=1,\dots, \#\mathcal I_n^p-3$, because the maximum intersection for non-overlapping intervals is $\varphi^{-n}$. Then define for all $n\geq 0$ 
$$\hat \C_n^p:=\bigcup_{i=1}^{\lfloor \#\mathcal I_n^p/3\rfloor} I_{n,3i}. $$
Note that $\hat \C^p_n\subseteq \C^p$ contains $\lfloor \#\mathcal I_n^p/3\rfloor\simeq \#\mathcal I_n^p$ connencted components of size $\varphi^{-n+1}$. Considering the limit set
$$\hat \C^p:=\bigcap_{n\geq 0} \hat \C_n^p$$
we then have that $\hat \C^p\subseteq \C^p$. Then one can deduce  $\EE(\dim_{BOX} \hat \C^p)=\EE(\dim \hat \C^p)$ (see for instance \cite{falbook} and, consequently
\begin{align*}\lim_{n\to\infty}\frac{\EE_n}{n \log\varphi}&=\lim_{n\to\infty}\frac{\EE(\#\mathcal I_n)}{(n-1)\log\varphi}\\
&=\EE(\dim_{BOX} \hat \C^p)=\EE(\dim \hat \C^p)\leq \EE(\dim_{BOX} \C^p).\end{align*}
This concludes the proof of \eqref{e1} and moves the problem to the estimation of $\EE_n$. To conclude the proof of the theorem we need to prove 
\begin{equation}\label{e2}\log \EE_n\sim n  \log(\min \{(2p),\varphi\})\quad \text{as $n\to\infty$}.\end{equation}

 We clearly have $\# \mathcal I_n\leq 2^n$, hence 
\begin{align*}\EE_n^p&=\sum_{I\in \mathcal I^p_n}\PP(I \text{ is retained})\leq (2p)^n
\end{align*}
On the other hand, by Proposition \ref{p1}, the cardinality of $\mathcal I_n^p$ is bounded by $F_{n+3}-1\sim \varphi^{n}$, hence 
\begin{equation}\label{e3}\EE_n\leq \min \{(2p)^n,F_{n+3}-1\}\leq \min \{(2p)^n,\varphi^{n+1}\}.\end{equation}

To prove a lower estimate, we write $\EE_n$ explicitly as a function of $p$:
\begin{align*}
     \EE_n &=\sum_{[\mathbf c]\in \mathbf G_n}\,\PP(\text {at least one }\mathbf d \in [\mathbf c]\text{ survives in the binary tree})\\
&= \sum_{[\mathbf c]\in \mathbf G_n}\,1-(1-p^n)^{\#[\mathbf c]}=:\EE_n(p)
\end{align*}


Using the partition of the set of greedy expansions $\mathbf G_n$ in \eqref{gn} we have
$$\EE_n(p)=\sum_{h=-1}^{n-1}\sum_{[\mathbf c]\in \mathbf G_n^h}\,1-(1-p^n)^{\#[\mathbf c]}.$$
By Proposition \ref{p2} we have
\begin{enumerate}[(i)]
\item if $\mathbf c\in \mathbf G_n^{-1}$ then by definition $\mathbf c=(0)^n$ and $\#[\mathbf c]=1$ for all $n$
\item if $\hat{\mathbf c}\in \mathbf G_n^{0}$ then $\hat {\mathbf c}=\mathbf c 1$ for some $\mathbf c\in  \mathbf G_{n-1}^h$ with $h\not=0$ (because greedy expansions cannot have occurrence of the subword $011$) and $\#[\hat{\mathbf c} ]=\#[\mathbf c]$;
\item if $\hat{\mathbf c}\in \mathbf G_n^{1}$ then $\hat {\mathbf c}=\mathbf c 0$ for some $\mathbf c\in \mathbf G_{n-1}^{0}$ and $\#[\hat{\mathbf c} ]=\#[\mathbf c]$;
\item if $\hat{\mathbf c}\in \mathbf G_n^{h}$ with $h\geq 2$ then $\hat {\mathbf c}=\mathbf c 00$ for some $\mathbf c\in \mathbf G_{n-2}^{h-2}$ and $\#[\hat{\mathbf c} ]\geq \#[\mathbf c]+1$. In particular, for all $h\geq 0$ and if $x\in [0,1]$ then 
$$1-(1-x)^{\#[\mathbf c00]-\#[\mathbf c]}\geq x.$$
\end{enumerate}
The above considerations allows us to establish the following recursive relation on $\EE_n(p)$:
\begin{align*}
\EE_n(p)=&\sum_{h=-1}^{n-1}\sum_{[\mathbf c]\in \mathbf G_n^h}\,1-(1-p^n)^{\#[\mathbf c]}\\
=&p^n+ \sum_{\mathbf c\in \mathbf G_n^0}\,1-(1-p^n)^{\#[\mathbf c]}+  \sum_{\mathbf c\in \mathbf G_n^1}\,1-(1-p^n)^{\#[\mathbf c]}
\\&+\sum_{h=2}^{n-1}\sum_{\mathbf c\in \mathbf G_n^0}\,1-(1-p^n)^{\#[\mathbf c]}\\
=&p^n+ \sum_{h=-1, h\not=0}^{n-2} \sum_{\mathbf c\in \mathbf G_{n-1}^h}\,1-(1-p^n)^{\#[\mathbf c 1]}+ \sum_{\mathbf c\in \mathbf G_{n-1}^0}\,1-(1-p^n)^{\#[\mathbf c0]}\\
&+
\sum_{h=0}^{n-3}\sum_{\mathbf c\in \mathbf G_{n-2}^{h}}\,1-(1-p^n)^{\#[\mathbf c00]}\\
=&p^n+ \sum_{h=-1}^{n-2} \sum_{\mathbf c\in \mathbf G_{n-1}^h}\,1-(1-p^n)^{\#[\mathbf c]}\\
&+
\sum_{h=0}^{n-3}\sum_{\mathbf c\in \mathbf G_{n-2}^{h}}\,1-(1-p^n)^{\#[\mathbf c00]}\\
=&p^n+\EE_{n-1}(p^\frac{n}{n-1})+
\sum_{h=0}^{n-3}\sum_{\mathbf c\in \mathbf G_{n-2}^0}\,1-(1-p^n)^{\#[\mathbf c00]}\\
=&\EE_{n-1}(p^\frac{n}{n-1})+
\sum_{h=-1}^{n-3}\sum_{\mathbf c\in \mathbf G_{n-2}^h}\,1-(1-p^n)^{\#[\mathbf c00]}\\
=&\EE_{n-1}(p^\frac{n}{n-1})+\EE_{n-2}(p^\frac{n}{n-2})+\sum_{h=-1}^{n-3}\sum_{[\mathbf c]\in \mathbf G_{n-2}^h}\,(1-p^n)^{\#[\mathbf c]}-(1-p^n)^{\#[\mathbf c00]}\\
=&\EE_{n-1}(p^\frac{n}{n-1})+\EE_{n-2}(p^\frac{n}{n-2})+\sum_{h=-1}^{n-3}\sum_{\mathbf c\in \mathbf G_{n-2}^h}\,(1-p^n)^{\#[\mathbf c]}(1-(1-p^n)^{\#[\mathbf c00]-\#[\mathbf c]})\\
\geq&\EE_{n-1}(p^\frac{n}{n-1})+\EE_{n-2}(p^\frac{n}{n-2})+p^n\sum_{h=-1}^{n-3}\sum_{\mathbf c\in \mathbf G_{n-2}^0}\,(1-p^n)^{\#[\mathbf c]}.
\end{align*}
Recalling that summing the cardinality of all the congruence classes of lenght $n-2$ returns by construction the cardinality of the initial set $\{0,1\}^{n-2}$, we deduce $$\sum_{h=-1}^{n-3}\sum_{[\mathbf c]\in \mathbf G_{n-2}^0}\,{\#[\mathbf c]}=2^{n-2}$$
and, since $(1-p^n)^{\#[\mathbf c]}\geq \#[\mathbf c](1-p^n)$,
\begin{equation}\label{eqrec}
\EE_n(p)\geq \EE_{n-1}(p^\frac{n}{n-1})+\EE_{n-2}(p^\frac{n}{n-2})+{2^{n-2}}p^n(1-p^n).
\end{equation}
Applying the above inequality to $\hat n=n-k$ for some $k=0,\dots,3$ and $\hat p=p^\frac{n}{n-k}$ we deduce
\begin{equation}\label{eqreck}\EE_{n-k}(p^\frac{n}{n-k})\geq \EE_{n-k-1}(p^\frac{n}{n-k-1})+\EE_{n-k-2}(p^\frac{n}{n-k-2})+{2^{n-k-2}}p^n(1-p^n).
\end{equation}
Setting 
 $$b_{n,k}:={2^{n-k-2}}p^n(1-p^n)$$ 
 and, iterating \eqref{eqrec}, we end up with 
 $$
 \EE_n(p)\geq F_{n-1}\EE_{2}(p^\frac{n}{2})+F_{n-2}\EE_{1}(p^n)+\sum_{k=0}^{n-3}b_{n,k}.
 $$
Now, by a direct computation $\EE_2(p^{n/2})=4p^n=2\EE_1(p^n)$. Since $4F_n\geq \varphi^n$ for all $n$, then $n\geq 3$
\begin{align*}
 \EE_n(p)&\geq 4p^n F_{n} + p^n(1-p^n)\sum_{k=3}^{n}2^{k-2}\\
 &= 4p^n F_{n} +  p^n(1-p^n)(2^{n-1}-2)\\
 & \geq (2^{n-1}-2) p^n+p^n \left(\varphi^n -  (2^{n-1}-2) p^n\right).
\end{align*}
If $p\leq \varphi/2$ then we readily get by above inequality 
\begin{equation}\label{e10}
 \EE_n(p)\geq  (2^{n-1}-2) p^n.
\end{equation}
If otherwise $p> \varphi/2$ then $p^n\in[\varphi^n/2^n,1]$ and
\begin{equation*}
 \EE_n(p)\geq \min_{x\in [\varphi^n/2^n,1]}  (2^{n-1}-2) x(1+\frac{\varphi^n}{2^{n-1}-2}-x)>\frac{2^{n-1}-2}{2^n}\varphi^n.
\end{equation*}
Putting together \eqref{e10} and the above inequality, we conclude
\begin{equation*}
 \EE_n(p)\geq \frac{2^{n-1}-2}{2^n} \min\{2^{n} p^n,\varphi^n\}
\qquad \forall p\in[1/2,1]\end{equation*}
which, together with \eqref{e3}, implies \eqref{e2} and, finally that $$\EE(p)=\min\{\log(2p)/\log(\varphi),1\}.$$

\medskip The last part of the proof concerns the topology of $\C$. We have that $\C$ contains an interval if and only if an interval $I\in \mathcal I_N^p$ belongs to $\C_n$ for all $n\geq N$. This means that $I$ is retained for infinitely many iterations, an event that has null probability --- remark that if $p\in [1/2,\varphi/2]$ this can be trivially deduced from the  $\C$ has a non-integer Hausdorff dimension. 
\end{proof}
\section{Conclusions and perspectives}\label{s4}
In this paper, we addressed the topology and the dimension of a random Cantor set with overlaps, generated by an IFS with scaling ratio equal to $\varphi$. The results extend known formulas to a case where OSC fails. Due to the role of $\varphi$ as a threshold for the cardinality of unique representations in non-integer base, we believe that the present study paves the way to the morphology of random Cantor sets at least in the case of a scaling ratio $q\in(1,\varphi)$, with possible extensions to $q\in(1,2)$. 

We also point out that the extension to the higher-dimensional case deserves an ad hoc study, because a random Cantor dust is a random process distinct from considering the cartesian product of unidimensional random Cantor sets, being the latter a random process in which the base and the height of each subrectangle are independent, see for instance \cite{falgri92}. However we believe that this extension would be worthwhile of investigation because it would allow to look into percolation processes in this redundant scenario. 
\bibliographystyle{alpha}
\bibliography{biblio_cantor}
\appendix
\section{Proof of Proposition \ref{p1}}\label{app1}
\begin{enumerate}[i)]
\item By above argument, we can choose as class representatives the greedy expansions in base $\varphi$ of length $n$. Then the problem is equivalent to establishing the cardinality of $\mathbf G_n$. Now 
$\mathbf G_n^H:=\{\mathbf c\mid \text{the last digit of $\mathbf c$ is $0$}\}$ so that $\#\mathbf G_n=1+\#\mathbf G_n^0+\#\mathbf G_n^H$. Now, if $\mathbf c\in \mathbf G_n$, then  $\mathbf c c\in \mathbf G_{n+1}$ for any $c\in\{0,1\}$ unless the resulting word $\mathbf c c$ ends with $011$, that is forbidden in greedy expansions.  In particular, if $\mathbf c\in \mathbf G_n^0$ then $\mathbf c c\in\mathbf G_{n+1}$ if and only if $c=0$. Then we have the recursive system: 
$$\begin{cases}\#\mathbf G_{n+1}=1+\# \mathbf G_{n}^0+2\# \mathbf G_{n}^H\\
\#\mathbf  G_{n+1}^0=\# \mathbf G_{n}^H\\
\#\mathbf  G_{n+1}^H=\#\mathbf  G_{n}^0+\#\mathbf G_{n}^H\\
\end{cases}$$
with intial conditions $\# \mathbf G_{1}^0=0,\#\mathbf G_{1}^H=1,\#\mathbf G_{1}=2$, whose solution is $\mathbf G_{n}=F_{n+3}-1$,$\mathbf G^H_{n}=F_{n+2}-1$,$\mathbf G^0_n=F_{n+1}$.
\item Immediate from the definition of $I_\mathbf c$.
\item  Let $I_\mathbf c\cap I_\mathbf d\not=\emptyset$. Assume, without loss of generality, that $\mathbf c$ and $\mathbf d$ are greedy expansions and that the left endpoint of $I_\mathbf d$ falls into $I_\mathbf c$. This implies that 
$$\delta_n:=\sum_{i=1}^{n}\frac{d_i-c_i}{\varphi^i}\in [0,\frac{1}{\varphi^{n-1}}]. $$
Since $\mathbf d$ and $\mathbf c$ are both greedy, then $\delta_n\geq 0$ implies that $\mathbf d\geq \mathbf c$ in the lexicographic sense \cite{parry1960}, namely if $n_0$ is the smallest index such that $d_{n_0}\not=c_{n_0}$ then $d_{n_0}>c_{n_0}$. In particular $d_{n_0}=1$ and $c_{n_0}=0$. By the arbitrariness of $n$, we can assume that $n_0=1$. Again by \cite{parry1960} we have if $x_k=0$ for some greedy expansion $(x_i)$, then $\sum_{i=k+1}^\infty x_i/\varphi^i<1/\varphi^{k}$. When applied to $\mathbf c$, this implies that 
$$\sum_{i=2}^n \frac{c_i}{\varphi^i}<\frac{1}{\varphi}.$$
Therefore we have
$$\frac{1}{\varphi^{n-1}}>\frac{1}{\varphi}+\sum_{i=2}^n \frac{d_i}{\varphi^i}-\sum_{i=2}^n \frac{c_i}{\varphi^i}>\sum_{i=2}^n \frac{d_i}{\varphi^i}$$
and this implies that $$\sum_{i=2}^n \frac{d_i}{\varphi^i}=\frac{d_{n}}{\varphi^n}.$$ 

Moreover, since $\delta_n\in[0,1/\varphi^{n-1}]$, then 
$$\sum_{i=2}^n \frac{c_i}{\varphi^i} \in\left[\frac{1}{\varphi}-\frac{1}{\varphi^{n-1}},\frac{1}{\varphi}\right).$$

Now, we claim that $$c_1\dots c_n=\bar c_1\cdots c_n:=\begin{cases}
(01)^{n/2-1}0c_n &\text{if $n$ is even}\\
(01)^{(n-1)/2-1}0c_{n-1}c_n&\text{if $n$ is odd}.
\end{cases}.$$

In order to seek a contradiction, assume that there exists $m\leq n-2$ such that $c_m\not=\bar c_m$. As $\mathbf c$ is greedy, the block $011$ is forbidden, consequently $m$ must be even and $c_m=0$. As remarked above, \cite{parry1960} implies that $\sum_{i=m+1}^n c_i/\varphi^{i}<1/\varphi^m$. Then we have
\begin{align*}\frac{1}{\varphi}-\frac{1}{\varphi^{n-1}}&\geq
\sum_{i=2}^{n}\frac{c_i}{\varphi^i}=\frac{1}{\varphi^2}+\cdots+\frac{1}{\varphi^{m-2}}+\sum_{i=m+1}^n \frac{c_i}{\varphi^i}\\
&<\frac{1}{\varphi^2}+\cdots+\frac{1}{\varphi^{m-2}}+\frac{1}{\varphi^m}=\frac{1}{\varphi}-\frac{1}{\varphi^{m+1}}
\end{align*}
Then $m+1>n-1$ and this contradicts $m\leq n-2$. Let $\bar n:=n$ if $n$ is even and $\bar n:=n-1$ if $n$ is odd. Define $c_{n+1}=0$. We deduce that 
$$\sum_{i=1}^n \frac{c_i}{\varphi^i}=\frac{1}{\varphi}-\frac{1}{\varphi^{\bar n-1}}+\frac{c_{\bar n}}{\varphi^{\bar n}}+\frac{c_{\bar n+1}}{\varphi^{\bar n+1}}.$$
and, consequently,
\begin{align*}
|I_{\mathbf c}\cap I_{ \mathbf d}|=\sum_{i=1}^n\frac{c_i}{\varphi^i}+\frac{1}{\varphi^{n-1}}-\sum_{i=1}^n\frac{d_i}{\varphi^{i}}=\frac{1}{\varphi^{ n-1}}-\frac{1}{\varphi^{\bar n-1}}+\frac{c_{\bar n}}{\varphi^{\bar n}}+\frac{c_{\bar n+1}}{\varphi^{\bar n+1}}-\frac{b_n}{\varphi^n}.
\end{align*}
By distinguishing the cases according to the parity of $n$ and appropriate binary values of $c_{n-1},c_{n}$ and $b_n$, we obtain the claim. If $n$ is even, then $\bar n=n$ and the above equality implies $c_n\geq b_n$ and 
$$|I_\mathbf c\cap I_\mathbf d|=\frac{c_n-b_n}{\varphi^n}\in\left\{0,\frac{1}{\varphi^n}\right\}$$
If otherwise $n$ is odd, then $\bar n=n-1$ and 
$$|I_\mathbf c\cap I_\mathbf d|=-\frac{1}{\varphi^{n}}
+\frac{c_{n-1}}{\varphi^{n-1}}+\frac{c_{n}}{\varphi^{n}}-\frac{b_n}{\varphi^n}.$$
By a direct computation, restricting the evaluation to the $c_n, c_{n-1},b_n$ ensuring a positive handside in the above equality, (say for instance $c_{n-1}=1,c_{n}=d_n=0)$) we obtain the claim. 
\end{enumerate}-
\end{document}